\documentclass[submission,nomarks]{dmtcs-episciences}

\usepackage[utf8]{inputenc}
\usepackage{subfigure}
\usepackage[utf8]{inputenc}
\usepackage[english]{babel}
\usepackage{amsmath, amsthm, amssymb}
\usepackage[dvipsnames]{xcolor}

\usepackage[round]{natbib}

\author{Jesse Racicot
  \and Giovanni Rosso\thanks{Partly funded by the FRQNT grant 2019-NC-254031 and the NSERC grant RGPIN-2018-04392.}
  }
\title[Domination in Kn{\"o}del graphs]{Domination in Kn{\"o}del graphs}
\affiliation{
  Concordia Univeristy, Qu\'ebec, Canada}
\received{2021-02-05}

\revised{2021-10-26,2022-03-18}

\accepted{2022-03-30}

\usepackage{todonotes}

\newtheorem{theorem}{Theorem}[section]
\newtheorem{conjecture}[theorem]{Conjecture}
\newtheorem{proposition}[theorem]{Proposition}
\newtheorem{corollary}[theorem]{Corollary}
\newtheorem{lemma}[theorem]{Lemma}
\newtheorem{definition}[theorem]{Definition}

\newtheorem{question}{Question}

\theoremstyle{definition}
\newtheorem{case}{Case}

\begin{document}
\publicationdetails{24}{2022}{1}{16}{7158}
\maketitle
\begin{abstract}
  Given a graph $G$ and an integer $k$, it is an $NP$-complete problem to decide whether $G$ has a dominating set of size at most $k$. In this paper we study this problem for the Kn{\"o}del Graph on $n$ vertices using elementary number theory techniques. In particular, we show an explicit upper bound for the domination number of the Kn{\"o}del Graph on $n$ vertices any time that we can find a prime number $p$ dividing $n$ for which $2$ is a primitive root.
\end{abstract}

\section{Introduction}

Given a graph $G = (V, E)$, a subset $D \subseteq V$ is said to be a \emph{dominating set} if every vertex in $V$ is in $D$ or adjacent to some vertex in $D$.  A dominating set of minimum size is called a $\gamma$-set and the size of any $\gamma$-set is denoted by $\gamma(G)$.  The problem of finding a minimum dominating set is a computationally difficult optimization problem.  In particular, given a graph $G$ and an integer $k$, determining whether $\gamma(G) \leq k$ is $NP$-complete \cite{NPComplete}.

The Kn{\"o}del graph was implicitly defined in \cite{KnodelGossip}.  Therein, Walter Kn{\"o}del answers the following question:  Given $n$ people, each with a unique message they wish to share with the others where sharing the information with a peer requires one discrete time unit, what is the minimum number of time units required so that every person knows every message?  Kn{\"o}del describes a protocol, referred to as \emph{gossiping} in the literature, which gives rise to the structure of the Kn{\"o}del graph, e.g., the people are vertices and for any given person, the list of people they inform throughout the protocol are their neighbors in the graph.  The Kn{\"o}del graph has been a topic of interest since and the interested reader can see \cite{KnodelSurvey} for an in-depth survey.  Here we present the definition of the Kn{\"o}del graph used by Bermond et al., see \cite{ModifiedKnodelDef} which is equivalent to the original definition of the Kn{\"o}del graph.  We should mention that this particular definition originally appeared in \cite{OriginalModifiedKnodelDef}. 
All logarithms in this paper are in base $2$.
\begin{definition}\label{def : Knodel Graph}
Let $n \geq 6$ be even and let $KG_{n} = (V, E)$ denote the Kn{\"o}del graph on $n$ vertices where 

$$ V = \{ 0, 1, 2, ..., n - 1 \} $$ 

and

$$ E = \{ \{ x, y \} \mid x + y \equiv 2^{t} - 1 (\bmod \: n) \} $$

$$ where \: t = 1, 2, ..., \lfloor \log n \rfloor. $$

\end{definition}
Observe that every vertex in $KG_n$ has degree $\lfloor \log n \rfloor$ although it is worth mentioning that there is a more general definition of the Kn{\"o}del graph.  By taking some integer $1 \leq k \leq \lfloor \log n \rfloor$, possibly dependent on $n$, we can allow the value $t$ (given in the preceding definition) to range from $1, 2, ..., k$ so that one obtains a graph where every vertex has degree $k$.

Domination in Kn{\"o}del graphs has been studied in \cite{mojdeh2018domination, xueliang2009domination} for the special case where the graph has small constant degree.  The study of different variants of domination have also appeared in \cite{mojdeh2019domination, varghese2018power, jafari2021total}.  The Kn{\"o}del graphs are of particular interest in the area of \emph{broadcasting}, a topic closely related to gossiping, since $KG_n$ is known to be a broadcast graph \cite{ModifiedKnodelDef}.  Within \cite{HHandAL}, the authors provide an explicit application of dominating sets to broadcasting. In particular, for a given value of $n$ they construct a sparse broadcast graph on $n + 1$ vertices by finding a dominating set of $KG_{n}$ satisfying certain properties and joining an additional vertex to every vertex in this dominating set.  In fact, the results of Section~\ref{sec:upperbounds} are heavily inspired by said paper.

The following paper is centered around domination in $KG_n$ for various values of $n$.  Our main results are Theorems~\ref{thm:1} and \ref{thm:2}.  We fix a prime number $p$ dividing $n$ and we suppose that $2$ is a primitive root modulo $p$ (see Definition~\ref{def:prim root}); in the first theorem we prove that if $p \leq  \lceil \log n \rceil$ then  $\gamma(KG_n) \leq \frac{n}{p}$. 
In the second theorem we suppose that $2$ is a primitive root modulo $p^k$, where $p^k$ divides $n$ and $\phi(p^k) < \lceil \log n \rceil$, for Euler totient function $\phi$, see the introduction of Section \ref{sec:upperbounds}. Then we show that $\gamma(KG_n) \leq \frac{2n}{p^k}$.
Both results are constructive, as we exhibit an explicit dominating set. \\

Note that it is important to have both results, as it is not expected that $2$ is a primitive root modulo $p^2$ any time that $2$ is a primitive root modulo $p$. Wieferich primes are defined as primes for which $p^2 \mid (2^{p-1}-1)$. Only Wieferich primes have been found, after checking all primes smaller than $10^{15}$: $1093$ and $3511$.  It is conjectured that there are infinitely Wieferich primes and that they are very rare. This conjecture is motivated by the heuristic explained below; in particular, the conjectured  density is $\frac{\log(\log(X))}{X}$. This means that, given a positive real number $X$, among all prime numbers smaller than $X$, roughly $\log(\log(X))$ are Wieferich primes.

Note that for neither of the two known Wieferich primes $2$ is a primitive root, but we now explain why it is reasonable to expect that Wieferich primes for which $2$ is a primitive root modulo $p$ are infinite; we thank A.~Granville for explaining the following heuristic to us.

We make the following two assumptions: that the chances of $2$ being a primitive root modulo $p$ is uniformly distributed ({\it i.e.} it is $\phi(p-1)/(p-1)$); and that when $2$ is a primitive root modulo $p$ then $(2^{p-1}-1)/p$ is also uniformly distributed modulo $p$. Note that if $(2^{p-1}-1)/p$ is $0$ modulo $p$, then $p$ is  a Wieferich prime and $2$ is not a primitive root modulo $p$. 
Hence, if we assume also that these two events are independent, the probability that $2$ is a primitive root modulo $p$ and that $p$ is a Wieferich prime is then 
\[
\frac{\phi(p-1)}{p(p-1)}. 
\]

Let us denote by $\pi(X)$ the number of primes smaller than $X$, for $X$ a positive real number. We thank again A.~Granville for the proof of the following lemma:
\begin{lemma}
For every $\delta >0$ there is $c >0$ such that $\phi(p-1)/(p-1) <c$ for only $\delta \pi(X)$ primes of size smaller than $X$ and then we have
\[
\sum_{p<X} \frac{\phi(p-1)}{p(p-1)} > c(1-\delta)\sum_{p<X} \frac{1}{p}= c' \log(\log(X)) + O(1).
\]
\end{lemma}

\begin{proof} 
Note that \[\frac{n}{\phi(n)} = \sum_{d\mid n} \frac{\mu^2(d)}{\phi(d)} \leq 2 \sum_{d|n, d<n^{1/2+\epsilon}} \frac{\mu^2(d)}{\phi(d)}\]
and so 
\[\sum_{p \leq x} \frac{(p-1)}{\phi(p-1)} \leq  2 \sum_{d<x^{1/2+\epsilon}} \frac{\mu^2(d)}{\phi(d)} \sum_{p \leq x, p \equiv 1 \bmod d} 1\]
Then by Brun--Titchmarsh theorem this is less than
\begin{align*}
    (8+\epsilon) \sum_{d<x^{1/2+\epsilon}} \frac{\mu^2(d)}{\phi(d)} \frac{\pi(x)}{\phi(d)} < & (8+\epsilon)  \pi(x) \sum_{d \geq 1} \frac{\mu^2(d)}{\phi(d)^2} \\
    = (8+\epsilon)  \pi(x) \prod_p (1+1/(p-1)^2) < 23 \pi(x).
\end{align*}
Therefore $\left\{ p \leq x: \phi(p-1)/(p-1)< c\right\} < 23 c \pi(x)$. 
\end{proof} 

As the latter sum $\sum_{p<X} \frac{1}{p}$ diverges, being the value of Riemann $\zeta$ function at $1$, the former diverges too.  As the sum $\sum_{p<X} \frac{\phi(p-1)}{p(p-1)} $ should approximate the number of primes $p$ smaller than $X$ for which $2$ is a primitive root modulo $p$ and $p$ is a Wieferich prime, the divergence of this sum hints to an infinitude of such primes.\\

The paper is organized in the following manner.  Section $2$ is the core of the paper and presents both the state-of-the-art results on $\gamma(KG_n)$ and our improved upper bounds on $\gamma(KG_n)$ for a very general class of even values of $n$.  
Section $3$ will discuss necessary conditions to achieve the best theoretical lower bound on $\gamma(KG_n)$ and Section 4 concludes the paper with conjectures about $\gamma(KG_n)$ as well as possible directions for further research.

\section{Upper bounds on $\gamma(KG_n)$}\label{sec:upperbounds}

The following section presents some upper bounds on $\gamma(KG_n)$ that apply to a large class of even integers.  We state a few definitions and preliminary results from elementary number theory as they will be heavily used in the arguments that follow.

Let $n > 1$ be a positive integer.  The number of positive integers less than $n$ that are relatively prime to $n$ is given by \emph{Euler's totient function}, denoted by $\phi(n)$.  If $\mathrm{gcd}(a, n) = 1$, it is well known that $a^{\phi(n)} \equiv 1 \: (\bmod \: n)$ but $\phi(n)$ is not necessarily the smallest integer for which this congruence holds.  Thus, we define the \emph{order} of $a$ modulo $n$ as the smallest positive integer $k$ such that $a^k \equiv 1 \: (\bmod \: n)$.  This prompts the following definition.

\begin{definition}\label{def:prim root}
Let $n > 1$ be a positive integer and let $a$ be an integer such that $\mathrm{gcd}(a, n) = 1$.  If a has order $\phi(n)$ modulo $n$, then $a$ is said to be a primitive root modulo $n$.  
\end{definition}

Note that most integers don't have primitive roots; a primitive root exists only when $n=p^{k}$ or $n=2p^{k}$, with $p$ an odd prime, or $n=2,4$. 

If we let $a_1, a_2, ... , a_{\phi(n)}$ be the positive integers less than $n$ and relatively prime to $n$ then whenever $a$ is a primitive root of $n$, we have that $a, a^2 , ..., a^{\phi(n)}$ are congruent modulo $n$ to $a_1, a_2, ... , a_\phi(n)$ in some order.  Of particular usefulness for this paper is the fact that $\phi(p) = p - 1$ and, more generally, $\phi(p^k) = p^k - p^{k-1}$ for any prime $p$ and $k \geq 1$.

With the preliminaries out of the way we are ready to investigate $\gamma(KG_n)$.  We state the following  result proved in \cite{HHandAL}.

\begin{theorem}\label{thm: KnodelOriginal}
Let $n$ be even such that $\lceil \log{n} \rceil = p$ where $p$ is an odd prime.  Moreover, suppose that $p$ divides $n$ and that $2$ is a primitive root modulo $p$. It then follows that 

$$ \gamma(KG_n) = \frac{n}{p}. $$

\end{theorem}

Given the conditions in the hypothesis, the authors construct a dominating set of size $\frac{n}{p}$ which yields that $\gamma(KG_n) \leq \frac{n}{p}$.  By remarking that the maximum degree of $KG_n$, denoted by $\Delta(KG_n)$ is $ \lfloor \log n \rfloor$ and applying Theorem~\ref{thm:3} (see Section~\ref{sec: lowerbounds}) which states that $\gamma(G) \geq \bigg \lceil \frac{n}{\Delta(G) + 1} \bigg \rceil$ for any graph $G$ on $n$ vertices, they obtain that $\gamma(KG_n) \geq \frac{n}{\lfloor \log{n} \rfloor + 1} = \frac{n}{\lceil \log{n} \rceil} = \frac{n}{p}$ (implicitly here, we have that $\lfloor \log{n} \rfloor + 1 = \lceil \log{n} \rceil = p$ which follows because $n$ is not a power of two since an odd prime $p$ divides $n$).  Thereby establishing $\gamma(KG_n)$ exactly.  Thus, the conclusion is as strong as one can hope for, although the conditions in the hypothesis are rather restrictive.  In fact, the best known upper bound on $\gamma(KG_n)$ for arbitrary even $n$, also given in \cite{HHandAL} is stated in the following theorem.


\begin{theorem}
 
For arbitrary even $n$, $\gamma(KG_n) \leq \frac{n}{4}.$

\end{theorem}

We will generalize the results of Theorem~\ref{thm: KnodelOriginal}.  In particular, we will relax some of the conditions on the value of $n$ and obtain positive results for $\gamma(KG_n)$.  Our first main result of the section is given below.  It establishes an upper bound of $\gamma(KG_n)$ whenever $n$ has an odd prime factor $p \leq \lceil \log{n} \rceil$ such that $2$ is a primitive root modulo $p$. Examples of such primes are $3,5,11,13, 19$. See Sequence $A001122$ in \cite{OEIS}.
Notice that we have relaxed the condition in Theorem~\ref{thm: KnodelOriginal} that the prime $p$ be equal to $\lceil \log{n} \rceil$.  Although this result does not apply to all even values of $n$, it applies to a rather general class of even integers.

\begin{theorem}\label{thm:1}
Let $n$ be even and suppose that $n$ has an odd prime factor $p \leq \lceil \log{n} \rceil$ such that $2$ is a primitive root modulo $p$. It then follows that 

$$ \gamma(KG_n) \leq \frac{n}{p}. $$

\end{theorem}

\begin{proof}
Notice that $\frac{n}{2p}$ is indeed an integer because $p$, an odd prime, is assumed to be a factor of $n$, where $n$ is also an even number. Thus, we consider the following set $D = \{ 2pl \mid 0 \leq l \leq \frac{n}{2p} - 1 \} \cup \{ 2pl - 1 \mid 1 \leq l \leq \frac{n}{2p} \}$.  We will argue that $D$ is a dominating set in $KG_n$ and given that the size of $D$ is $\frac{n}{2p} + \frac{n}{2p} = \frac{n}{p}$, the result will then follow.

To show that $D$ is a dominating set we will show that any vertex $x \notin D$ is adjacent to a vertex in $D$ with a constructive argument.  That is, we will give a closed form expression for the neighbour in $D$ which depends on the value of $x$.

Let $x \in V \setminus D$ be an arbitrary vertex.  Since $x \notin D$, it follows that $x$ takes the form $x = 2pl_0 + m_0$, where $0 \leq l_0 \leq  \frac{n}{2p} - 1$ and $1 \leq m_0 \leq 2p - 2$.  We break the proof into two cases, based on the parity of $x$.  

First consider the case where $x$ is odd.  Since $x$ is odd we must have that $m_0$ is odd which implies that $m_0 \neq p - 1$.  That is, $1 \leq m_0 \leq 2p - 2$ and $m_0 \neq p - 1$.  Therefore, we have that $\mathrm{gcd}(m_0 + 1, p) = 1$.  Using the fact that $2$ is a primitive root modulo $p$ we obtain that $ m_0 + 1 \equiv 2^i \: (\bmod \: p)$ for some $1 \leq i \leq p - 1$. 

That is, $2^i - 1 = m_0 + jp$ where $j$ is even because $2^i - 1$, $m_0$ and $p$ are all odd. Thus, consider $l_1 = \frac{j}{2} - l_0$ and $l_2 = \frac{n}{2p} + \frac{j}{2} - l_0$.  Notice that both $l_1$ and $l_2$ are integers because both $\frac{n}{2p}$ and $\frac{j}{2}$ are integers by previous remarks made. Also, notice that at least one of $l_1$ and $l_2$ is between $0$ and $\frac{n}{2p} - 1$, as $i < p \leq m$.  

If  $0 \leq l_1 \leq \frac{n}{2p} - 1$ then we take $s = 2pl_1 \in D$.  We have that $x + s = (2pl_0 + m_0) + 2pl_1 = (2pl_0 + m_0) + 2p(\frac{j}{2} - l_0) = (m_0 + jp) = 2^i - 1$.  That is, $x + s = 2^i - 1$ and $s$ is therefore adjacent to $x$.  

Similarly, if $0 \leq l_2 \leq \frac{n}{2p} - 1$ then we take $s = 2pl_2 \in D$ and obtain that $x + s =  (2pl_0 + m_0) + 2pl_2 = (2pl_0 + m_0) + 2p(\frac{n}{2p} + \frac{j}{2} - l_0) = (m_0 + jp) + n = 2^i - 1 + n$.  That is, $x + s \equiv 2^i-1 \: (\bmod \: n)$ and $s$ is therefore adjacent to $x$.

Thus, in the case that $x$ is odd we have shown that there is a vertex in $D$ that is adjacent to $x$.

Now, consider the case where $x$ is even.  We have that $m_0$ must be even and therefore $\mathrm{gcd}(m_0, p) = 1$. Similarly we obtain that $ m_0 \equiv 2^i \: (\bmod \: p)$ for some $1 \leq i \leq p - 1$.

That is, $2^i = m_0 + jp$ where $j$ must be even.  Following the argument given above we consider $l_1 = \frac{j}{2} - l_0$ and $l_2 = \frac{n}{2p} + \frac{j}{2} - l_0$ and select $s = 2pl_1 - 1$ or $s = 2pl_2 - 1$ accordingly.  One of these must be a vertex in $D$ adjacent to $x$.
\end{proof}

We informally state an immediate consequence of this theorem.  Consider any even $n$ which has an odd prime factor that satisfies the aforementioned conditions.  One can select the largest such prime factor of $n$ to achieve the strongest result.  When a prime greater than $3$ with the desired properties is found we have established a better upper bound than $\frac{n}{4}$ for a rather general class of even integers.  The formal statement is given explicitly in the following corollary, which generalizes Theorem~\ref{thm: KnodelOriginal} of \cite{HHandAL}.

\begin{corollary}
Let $n$ be even and suppose that $n$ has a prime factor $p$ with $3 < p \leq \lceil \log n \rceil$ such that $2$ is a primitive root modulo $p$. It then follows that 

$$ \gamma(KG_n) \leq \frac{n}{p} < \frac{n}{4}. $$
\end{corollary}

We now turn to the second main result in this section. We present an upper bound on a slightly more restricted class of even integers.  It should be noted that, in certain cases, this upper bound is a much better bound. 

\begin{theorem}\label{thm:2}
Let $n$ be even, $p$ be an odd prime and $k \geq 2$ be an integer. Suppose that $\phi(p^k) < \lceil \log{n} \rceil$, that $p^k$ divides $n$, and that $2$ is a primitive root modulo $p^k$. It then follows that 

$$ \gamma(KG_n) \leq \frac{2n}{p^k}. $$
\end{theorem}

\begin{proof}

Let $D = \{ lp^k \mid 0 \leq l \leq \frac{n}{p^k} - 1\} \cup \{ lp^k - 1 \mid 1 \leq l \leq \frac{n}{p^k} \}$.  We will show that $D$ is a dominating set of $KG_n$ by following a similar proof as above which now considers a couple more cases.  

Let $x \in V \setminus D$ be an arbitrary vertex and note that $x$ takes the form $x = l_0p^k + m_0$, where $0 \leq l_0 \leq  \frac{n}{p^k} - 1$ and $1 \leq m_0 \leq p^k - 2$.  We consider the case where $x$ is odd and the details for the case where $x$ is even follow identically.  

\theoremstyle{plain}

Suppose that $x = p^kl_0 + m_0$ is odd. We therefore have that either $l_0$ is even and $m_0$ is odd or $l_0$ is odd and $m_0$ is even.

\begin{case} $l_0$ is even and $m_0$ is odd.

If $l_0$ is even and $m_0$ is odd then we have that $m_0 + 1$ is even and  $\mathrm{gcd}(m_0 + 1, p^k) = 1$ or $\mathrm{gcd}(m_0 + 1, p^k) = p^a$ for some $a > 0$. 

If $\mathrm{gcd}(m_0 + 1, p^k) = 1$ we therefore obtain that $ m_0 + 1 \equiv 2^i \: (\bmod \: p^k)$ for some $1 \leq i \leq p^k - p^{k-1}$.  That is to say, $m_0 = 2^i - 1 + jp^k$.  

Thus, consider $l_1 = -(j + l_0)$ and $l_2 = \frac{n}{p^k} - (j + l_0)$.  Depending on the size of $l_1$ or $l_2$ we select $s = l_1p^k$ or $s = l_2p^k$ as a vertex in $D$. When $s=l_1p^k$  we have
\[
x+s= 2^i-1
\] and obtain that $s$ is adjacent to $x$ (by definition of Kn\"odel graph, as $i < \phi(p^k) < \lceil \log{n} \rceil $); similarly when $s = l_2p^k$.

If $\mathrm{gcd}(m_0 + 1, p^k) = p^a$ with $a>0$ then 
$\mathrm{gcd}(m_0, p^k) = 1$  and $m_0 = 2^i +jp^k$. Similarly, we consider $l_1 = -(j + l_0)$ or $l_2 = \frac{n}{p^k} - (j + l_0)$ and pick $s = l_1p^{k} - 1$ or $s = l_2p^{k} - 1$ accordingly and we are done. 

\end{case}

\begin{case}
In the case that $l_0$ is odd and $m_0$ is even we have that $\mathrm{gcd}(m_0, p^k) = 1$ or $\mathrm{gcd}(m_0, p^k) = p^a$ for some $a > 0$.  

If $\mathrm{gcd}(m_0, p^k) = 1$ then 
$m_0 = 2^i +jp^k$. Consider $l_1 = -(j + l_0)$ or $l_2 = \frac{n}{p^k} - (j + l_0)$ and pick $s = l_1p^{k} - 1$ or $s = l_2p^{k} - 1$.

If $\mathrm{gcd}(m_0, p^k) = p^a$ with $a > 0$ then $\mathrm{gcd}(m_0 + 1, p^k) = 1$  and $m_0 = 2^i - 1 +jp^k$. Consider $l_1 = -(j + l_0)$ or $l_2 = \frac{n}{p^k} - (j + l_0)$ and pick $s = l_1p^{k}$ or $s = l_2p^{k}$. 

\end{case}

\end{proof}
Informally, we can discuss the power of the two previous theorems. Take $n$ to be even and consider $n = 2^{k_1}p_2^{k_2}...p_j^{k_j}$.  Find the largest prime $p_l$ that meets the conditions in Theorem~\ref{thm:1} and the largest value $p_h^{e_h}$ with $e_h \leq k_h$ that meets the conditions in Theorem~\ref{thm:2}.  These two primes, $p_l$ and $p_h$, may coincide.  We then have that $\gamma(KG_n) \leq \mathrm{min} \left\{ \frac{2n}{p_h^{e_h}}, \frac{n}{p_l} \right\}$.

It is worth comparing Theorems~\ref{thm:1} and \ref{thm:2} with all the previously known results.  First we consider Theorem~\ref{thm: KnodelOriginal}, established in \cite{HHandAL}.  For the values of $n$ for which their results apply, the authors achieved the strongest possible bound on $\gamma(KG_n)$.  Yet, one should mention that they have not necessarily provided results for an infinite family of values.  Indeed, a condition on $n$ is that $\lceil \log{n} \rceil$ is a prime with $2$ as a primitive root although it is not known whether there are infinitely many such primes (i.e. this is a special case of Artin's Conjecture \cite{ArtinConjecture}).  Thus whether their results apply to infinitely many even values of $n$ is conditional on the conjecture.  Although not necessarily the strongest bounds in some cases, the results that we have presented unconditionally apply to an infinite family of even values.  

Within \cite{xueliang2009domination}, the authors study the Kn{\"o}del graph with fixed degree $3$, e.g., by augmenting Definition~\ref{def : Knodel Graph} to allow $t$ to only take on values $1, 2$ or $3$.  For fixed degree $3$, the exact domination number is found to be $2\lfloor \frac{n}{8} \rfloor + c$ for $0 \leq c \leq 2$, where $c$ depends on $n \bmod 8$.  Similarly, the work of \cite{mojdeh2018domination} examines the Kn{\"o}del graph with fixed degree $4$.  The exact value for the domination number is proved to be $2\lfloor \frac{n}{10} \rfloor + c$ for $0 \leq c \leq 4$, where $c$ depends on $n\bmod{10}$.  Although expressed slightly differently, these two results are roughly equal to the trivial lower bound (See Section~\ref{sec: lowerbounds}).


In a sense, all the known results apply to a special case of the Kn{\"o}del graph.  The results in \cite{mojdeh2018domination, xueliang2009domination} apply to the special case of small constant degree (but any even value of $n$) whereas our results, along with those found in \cite{HHandAL}, apply only to particular values of $n$ (but the degree is $\lfloor \log n \rfloor$). 

\section{Necessary Conditions for $\gamma(KG_n) = \bigg \lceil \frac{n}{\lfloor \log n \rfloor + 1} \bigg \rceil$}\label{sec: lowerbounds}

As a preliminary, we state the following result originally proved by Berge \cite{Berge}.

\begin{theorem}\label{thm:3}
Let $G$ be a graph on $n$ vertices and let $\Delta$ denote the maximum degree of any vertex in $G$.  It then follows that 
$\gamma(G) \geq \bigg \lceil \frac{n}{\Delta + 1} \bigg \rceil$.
\end{theorem}

Roughly put, meeting this lower bound is the best one can hope for.  From the description of $KG_n$ we see that every vertex has degree $\lfloor \log n \rfloor$ and as was previously shown in \cite{HHandAL}, there are some sufficient conditions which allow $KG_n$ to meet this lower bound of $\bigg \lceil \frac{n}{\lfloor \log n \rfloor + 1} \bigg \rceil$.  Although this section could have been appropriately titled lower bounds on $\gamma(KG_n)$ we will see that these bounds are not much better than the preceding general bound.  In this section we instead investigate necessary conditions on the value of $n$ if $\gamma(KG_n) = \bigg \lceil \frac{n}{\lfloor \log n \rfloor + 1} \bigg \rceil$ which can very well be interpreted as lower bounds on the domination number.   

To begin we present a few definitions and well known results which will be of use.  A graph $G = (V, E)$ is said to be \emph{$k$-regular} for some integer $k \geq 0$ if every vertex has degree $k$.  A dominating set $D \subseteq V$ is called \emph{perfect} if every vertex in $V \setminus D$ has exactly one neighbour in $D$.  A dominating set $D \subseteq V$  is called \emph{efficient} if it is a perfect dominating set that is also independent, e.g., for any pair $u, v \in D$ we have $\{ u, v \} \notin E$.  We say that a set is a $\gamma$-set if it is a dominating set of minimal cardinality. The following result is attributed to Haynes, Hedetniemi, and Slater \cite{Fundamentals}.

\begin{theorem}\label{thm: Efficient-DS}
Let $G$ be a graph on $n$ vertices and suppose that $\frac{n}{\Delta + 1}$ is an integer.  If $\gamma(G) = \frac{n}{\Delta + 1}$ then every $\gamma$-set is an efficient dominating set.
\end{theorem}

This result has some interesting implications which are illustrated in the following proposition.  We believe that this proposition may already be known yet not stated in the literature explicitly.
\begin{proposition}
Let $G = (V, E)$ be a $k$-regular bipartite graph on $n$ vertices where $k \geq 2$ and suppose that $\frac{n}{k + 1}$ is an integer.  If $\gamma(G) = \frac{n}{k + 1}$ then $\frac{n}{k + 1}$ is an even integer.  Moreover, every $\gamma$-set $D$ of $G$ can be partitioned into two equal sized sets.
\end{proposition}
\begin{proof}
Let $X$ and $Y$ be a bipartition of $V$, that is, $X$ and $Y$ partition $V$ and for all $\{ u, v \} \in E$ we have $u \in X$ and $v \in Y$.  Let $|X| = x$ and $|Y| = y$ and notice that $x = y$.  One can realize this by counting the edges of $G$ in two ways. That is, $|E| = k|X|$ and $|E| = k|Y|$ since $G$ is a $k$-regular bipartite graph.  Hence, $kx = ky$ which implies that $x = y$.

Now let $D$ be any $\gamma$-set of $G$ and let $D_X = X \cap D$ and $D_Y = Y \cap D$.  That is, $D_X$ and $D_Y$ partition the vertices in $D$ with respect to the aforementioned bipartition of $V$.  Denoting $x_0 = |D_X|$ and $y_0 = |D_Y|$ we will show that $x_0 = y_0$ and the result will then follow.  
Note that $D$ is an efficient dominating set by Theorem~\ref{thm: Efficient-DS}.  Thus, for the $kx_0$ edges incident on the $x_0$ vertices in $D_X$ we count precisely $kx_0$ distinct vertices belonging to $Y$ that are not in $D_Y$.  Now, since every vertex in $Y$ is either in $D_Y$ or counted by these incident edges, we obtain that $y = y_0 + kx_0$.  Similar remarks yield that $x = x_0 + ky_0$.
Noting that $x = y$ we obtain that $y_0 + kx_0 = x_0 + ky_0$.  Since $k \geq 2$, standard algebraic manipulation yields that $x_0 = y_0$ and the result then follows.
\end{proof}

One obtains the immediate corollary by remarking that $KG_n$ is a $\lfloor \log n \rfloor$-regular bipartite graph that is partitioned based on the parity of the vertices, e.g., even vertices belong to one part and odd vertices belong to the other.

\begin{corollary}
Let $n$ be even and suppose that $\frac{n}{\lfloor \log n \rfloor + 1}$ is an integer.  If $\gamma(KG_n) = \frac{n}{\lfloor \log n \rfloor + 1}$ then  $\frac{n}{\lfloor \log n \rfloor + 1}$ is an even integer.  Moreover, every $\gamma$-set $D$ can be partitioned into two equal sized sets $D_E$ and $D_O$ consisting of even and odd vertices, respectively.  
\end{corollary}

In some sense, this next proposition generalizes the last one.  It is perhaps better understood by considering the contrapositive.  Roughly put, it states that if one were to meet this lower bound of $\lceil \frac{n}{k + 1} \rceil$ and this value is an odd integer, then division of $n$ by $k + 1$ must leave a small remainder.

\begin{proposition}

Let $G = (V, E)$ be a $k$-regular bipartite graph on $n$ vertices where $k \geq 2$ and suppose that $n = 2j(k+1) + r$ with $4 \leq r < k + 1$.  It then follows that $\gamma(G) > 2j + 1 = \big \lceil \frac{n}{k + 1} \big \rceil$.

\end{proposition}

\begin{proof}
Suppose for the sake of deriving a contradiction that $\gamma(G) = 2j + 1$ and let $D$ be a $\gamma$-set of $G$.  Consider $D_X = D \cap X$ and $D_Y = D \cap X$ where $X$ and $Y$ form a bipartition of $V$.

Since $|D| = 2j + 1$ it follows that at least one of $D_X$ or $D_Y$ has no more than $j$ vertices.  Without loss of generality, assume that $|D_X| \leq j$.  We will show that $D$ cannot possibly dominate all the vertices of $Y$.  

If $|D_X| = j$ then $D_X$ dominates at most $k|D_X| = kj$ vertices of $Y$.  Since there are $j + 1$ vertices in $D_Y$ we see that the number of vertices in $Y$ that $D$ dominates is at most $k|D_X| + (j + 1) = kj + (j + 1) = j(k + 1) + 1 < \frac{n}{2}$.  But we know that $|X| = |Y| = \frac{n}{2}$ from  previous remarks and thus the vertices of $Y$ are not dominated by $D$ contradicting the fact that $D$ is a dominating set.

It is clear from the argument that $D$ cannot be a dominating set whenever $|D_X| < j$.  Therefore, it must be that $\gamma(G) > 2j + 1 = \lceil \frac{n}{k + 1} \rceil$.
\end{proof}
\begin{corollary}
Let $n$ be even and let $k = \lfloor \log n \rfloor$.  If $n = 2j(k + 1) + r$ with $4 \leq r < k + 1$ then $\gamma(KG_n) > 2j + 1 = \big \lceil \frac{n}{\lfloor \log n \rfloor + 1} \big \rceil $. 
\end{corollary}

\section{Conclusions and Further Research}
Of course the main problem remains and that is to determine $\gamma(KG_n)$ for arbitrary even $n$.  First, we present a few conjectures in order from most likely to least likely.

The first two conjectures seem very likely to be true although the details to verify them may be slightly more cumbersome. In particular, these conjectures would allow one to relax the restrictions in Theorems~\ref{thm:1} and \ref{thm:2} which require that $p$ be a factor of $n$.  The dominating set is likely similar to the dominating sets in these theorems but instead each vertex would be translated by the remainder that $n$ leaves upon division by said prime $p$.

\begin{conjecture}

Let $n$ be even and let $p \leq \lceil \log n \rceil$ be an odd prime such that $2$ is a primitive root modulo $p$.  It then follows that

$$ \gamma(KG_n) \leq \Bigg \lceil \frac{n}{p} \Bigg \rceil. $$

\end{conjecture}

\begin{conjecture}

Let $n$ be even, $p$ be an odd prime and $k \geq 2$ be an integer.  If $\phi(p^k) < \lceil \log n \rceil$ and $2$ is a primitive root modulo $p^k$ then

$$ \gamma(KG_n) \leq \Bigg \lceil \frac{2n}{p^k} \Bigg \rceil. $$

\end{conjecture}

Lastly, the final conjecture is that $\gamma(KG_n)$ is in fact very close to the trivial lower bound in Theorem~\ref{thm:3}. Perhaps, with some small additive constant. Due to the symmetry of $KG_n$ we expect this to be true, yet believe that the dominating set would not be as neatly described as it were in Theorems~\ref{thm:1} and \ref{thm:2}. Moreover, the known results (where the exact domination number is established) seem to agree with this conjecture.

\begin{conjecture}

Let $n$ be even with $\lfloor \log n \rfloor = m$.  It then follows that

$$ \Bigg \lceil \frac{n}{m + 1} \Bigg \rceil \leq \gamma(KG_n) \leq \Bigg \lceil \frac{n}{m + 1} \Bigg \rceil + O(\log(n)). $$

\end{conjecture}

Finally, we leave some problems of interest whose answer may be insightful in determining $\gamma(KG_n)$ and perhaps even in the general problem of domination in graphs.
\begin{question}
The necessary conditions presented in Section~\ref{sec: lowerbounds} are based on general arguments for $k$-regular bipartite graphs.  Although these arguments are appreciated for their own sake they do not incorporate the algebraic structure of $KG_n$.  Can we strengthen the necessary conditions presented in Section $3$ or simply find better lower bounds on $\gamma(KG_n)$?
\end{question}

\begin{question}
The upper bounds presented in Section~\ref{sec:upperbounds} impose the restriction that an odd prime factor of $\lceil \log n \rceil$ must have $2$ as a primitive root.  Can we relax this constraint?  An idea may be to construct a dominating set based on the order of $2$ modulo some chosen prime $p$.
\end{question}

\begin{question}
Can some of the techniques used in this paper be applied to other classes of graphs that exhibit a similar algebraic structure?  
\end{question}

\acknowledgements
\label{sec:ack}
We thank Andrew Granville for help with the analytic number theory results and  the anonymous referees for their very valuable suggestions and corrections that  improved the results and the presentation of this paper.

\nocite{*}
\bibliographystyle{abbrvnat}
\bibliography{sample-dmtcs}
\label{sec:biblio}

\end{document}